\documentclass{amsart}
\usepackage{graphicx, layout, appendix, subcaption}
\usepackage{amssymb, amsthm, amsmath, amscd, amsfonts, 
mathrsfs, latexsym, mathtools}
\usepackage{placeins}
\usepackage{enumitem}
\usepackage{setspace}
\usepackage[
	setpagesize=false,
	colorlinks=true,
	linkcolor=blue,
	pdfencoding=auto,
	psdextra,
]{hyperref}
\usepackage{cleveref}
\usepackage[final]{microtype}
\usepackage{xcolor}
\hypersetup{
    colorlinks,
    linkcolor={red!50!black},
    citecolor={blue!50!black},
    urlcolor={blue!80!black}
}
\usepackage{ifthen}
\usepackage{amsxtra}
\usepackage{amstext}
\usepackage{pifont}
\usepackage[T1]{fontenc}
\usepackage{textcomp}
\usepackage{tikz-cd}
\setlist[enumerate,1]{label=(\arabic*), ref=(\arabic*)}
\setlist[enumerate,3]{label=(\roman*), ref=(\roman*)}
\theoremstyle{plain}
\newtheorem{thm}{Theorem}[section]
\newtheorem*{thm*}{Theorem}

\newtheorem{lem}[thm]{Lemma}

\newtheorem*{lem*}{Lemma}

\newtheorem{prop}[thm]{Proposition}

\newtheorem*{prop*}{Proposition}

\newtheorem{cor}[thm]{Corollary}

\newtheorem*{cor*}{Corollary}

\newtheorem*{claim*}{Claim}

\newtheorem*{conj*}{Conjecture}

\theoremstyle{definition}

\newtheorem*{defn*}{Definition}

\newtheorem*{ques*}{Question}

\newtheorem{exa}[thm]{Example}

\newtheorem*{exa*}{Example}

\theoremstyle{remark}
\newtheorem{rmk}[thm]{Remark}

\newtheorem*{rmk*}{Remark}

\numberwithin{figure}{section}
\numberwithin{table}{section}
\numberwithin{equation}{section}

\def \Supp {\mathrm{Supp}}

\def \CC {\mathbb{C}}

\def \PP {\mathbb{P}}
\def \QQ {\mathbb{Q}}
\def \RR {\mathbb{R}}

\def \ZZ {\mathbb{Z}}

\def \Ocal {\mathcal{O}}

\def \Xcal {\mathcal{X}}
\def \Ycal {\mathcal{Y}}

\def \hbar {\bar{h}}

\DeclareMathOperator{\Sing}{Sing}
\DeclareMathOperator{\pr}{pr}

\newcommand{\HC}{\operatorname{HC}}

\DeclareMathOperator{\Spec}{Spec}

\DeclareMathOperator{\res}{res}

\DeclareMathOperator{\Pic}{Pic}

\newcommand{\Cone}{\operatorname{Cone}}
\newcommand{\diff}{\mathrm{diff}}
\DeclareMathOperator{\Nef}{Nef}

\newcommand{\NEbar}{\overline{\mathrm{NE}}}

\DeclareMathOperator{\ord}{ord}

\DeclareMathOperator{\length}{length}
\title[Nef cones of nested Hilbert schemes]{Nef cone decompositions for nested Hilbert schemes of points on surfaces}
\author{Chiwon Yoon}
\address{Department of Mathematical Sciences, KAIST, 291 Daehak-ro, Yuseong-gu, Daejeon, 34141, Republic of Korea
}
\email{dbs7985@kaist.ac.kr}


\begin{document}
\begin{abstract}
Let $S$ be a smooth projective surface with $q(S)=0$. 
We give numerical criteria for the nef cones of $S^{[n,n+1]}$ and $S^{[1,n]}$ to decompose as sums of pullbacks of nef cones under their natural morphisms.
These criteria recover the known decompositions for the projective plane, Hirzebruch surfaces, and Picard rank one K3 surfaces, and apply to del Pezzo surfaces with $2\le K_S^2\le7$.
For a del Pezzo surface of degree one, we show that the corresponding pullback decompositions fail for both $S^{[n,n+1]}$ and $S^{[1,n]}$.
\end{abstract}
\maketitle

\section{Introduction}
\label{sec:introduction}

Let $S$ be a smooth projective surface over $\CC$ with $q(S)=0$, and let $S^{[n]}$ denote the Hilbert scheme of length-$n$ subschemes of $S$.
We consider the nested Hilbert schemes
\[
S^{[n,n+1]}
=
\{(Z\subset Z')\mid \length(Z)=n,\ \length(Z')=n+1\}
\]
and
\[
S^{[1,n]}
=
\{(p,Z)\in S\times S^{[n]}\mid p\in\Supp(Z)\}.
\]
These spaces admit natural morphisms to $S$ and to Hilbert or symmetric products of $S$, and hence give pullback nef classes. 
We determine when these classes generate the nef cone.

Ryan--Yang \cite{RY18} computed these nef cones for the projective plane, Hirzebruch surfaces, and general Picard rank one K3 surfaces; Dutta--Edwards--Raha \cite{DER25} treated Picard rank two Mori dream K3 surfaces for large $n$.
Here we express the decomposition in terms of the Mori cones $\NEbar(S)$ and $\NEbar(S^{[m]})$.

For $S^{[n,n+1]}$, let
\[
\res:S^{[n,n+1]}\to S,\qquad
p_b:S^{[n,n+1]}\to S^{[n]},\qquad
p_a:S^{[n,n+1]}\to S^{[n+1]}
\]
be the natural morphisms.
For $S^{[1,n]}$, let
\[
\pr_{\diff}:S^{[1,n]}\to S,\qquad
\pr_b:S^{[1,n]}\to S^{(n-1)},\qquad
\pr_a:S^{[1,n]}\to S^{[n]}
\]
be the corresponding morphisms.

\begin{thm}\label{thm:delpezzo-main}
Let $S$ be a del Pezzo surface with
$
2\le K_S^2\le7
$
and let $n\ge2$. 
Then
\[
\Nef(S^{[n,n+1]})
=
\res^*\Nef(S)
+
p_b^*\Nef(S^{[n]})
+
p_a^*\Nef(S^{[n+1]}),
\]
and
\[
\Nef(S^{[1,n]})
=
\pr_{\diff}^*\Nef(S)
+
\pr_b^*\Nef(S^{(n-1)})
+
\pr_a^*\Nef(S^{[n]}).
\]
\end{thm}

The degree-one case is different. 
Let $S$ be a degree-one del Pezzo surface and put $F=-K_S$.
In addition to the punctual class and the classes associated with the $240$ $(-1)$-curves \cite[Chapter~8]{dolgachev2012classical}, the Mori cone of $S^{[m]}$ contains the extremal ray $F_{[m]}$, obtained from a $g^1_m$ on a smooth anticanonical curve.
The Mori-cone description is given in \cite[Theorem~5.1]{BHL+16}; see also the correction \cite{BHL+16erratum}. 
This additional ray leads to the failure of the pullback decompositions.

\begin{thm}\label{thm:intro-degree-one}
Let $S$ be a del Pezzo surface of degree $1$ and let $n\ge2$. Then
\[
\res^*\Nef(S)
+
p_b^*\Nef(S^{[n]})
+
p_a^*\Nef(S^{[n+1]})
\subsetneq
\Nef(S^{[n,n+1]}),
\]
and
\[
\pr_{\diff}^*\Nef(S)
+
\pr_b^*\Nef(S^{(n-1)})
+
\pr_a^*\Nef(S^{[n]})
\subsetneq
\Nef(S^{[1,n]}).
\]
\end{thm}

We also consider the Picard-rank-two K3 cases of \cite{DER25}, whose wall-crossing methods are based on the framework developed in \cite{bayer2014mmp}.
When the Mori cone of the K3 surface is generated by two smooth rational curves, their description of the Mori cone of the Hilbert scheme fits directly into the criterion developed here.
For an elliptic K3 surface, the extremal curves are pencil curves, and the corresponding dual-cone argument is slightly different.

The paper is organized as follows.
Section~\ref{sec:background} introduces the divisor and curve conventions and establishes the required intersection formulas. 
Section~\ref{sec:nested} develops the criterion for $S^{[n,n+1]}$ and applies it to del Pezzo and K3 surfaces.
Section~\ref{sec:universal} treats the corresponding criterion for $S^{[1,n]}$. 
Section~\ref{sec:degree-one} proves the failure of the pullback decomposition for $S^{[n,n+1]}$ in degree one, and Section~\ref{sec:degree-one-universal} gives the analogous result for $S^{[1,n]}$.

\subsection*{Acknowledgments}
This work was supported by the Institute for Basic Science (IBS-R032-D1).

\section{Background}
\label{sec:background}

Throughout, $S$ is a smooth projective surface over $\CC$ with $q(S)=0$, and we set $\rho:=\dim N^1(S)_\RR$. 
Write $\NEbar(Y)$ for the closed cone of effective curve classes in $N_1(Y)_\RR$. Since $q(S)=0$, we use Fogarty's description of the Picard groups of Hilbert schemes~\cite{Fogarty1973AlgebraicFO}.

\subsection{Hilbert schemes of points}
\label{subsec:hilbert-background}
For $m\ge2$, let $\HC_m:S^{[m]}\longrightarrow S^{(m)}$ be the Hilbert--Chow morphism.
For $L\in\Pic(S)$, let $L^{(m)}$ be the divisor class on $S^{(m)}$
descended from the symmetric external product on $S^m$.
Set $L^{[m]}:=\HC_m^*L^{(m)}$.
Let $B^{[m]}$ be the Hilbert--Chow boundary divisor. 
We follow the notation of \cite[Section~4.1]{RY18} for the following curve classes.

Let $\gamma\subset S$ be a reduced irreducible nodal curve, and set
$\delta(\gamma):=\#\Sing(\gamma)$.
Assume $m>\delta(\gamma)$.
Choose distinct points $q_1,\dots,q_{m-1}\in\gamma$ containing all nodes of $\gamma$, and define
\[
C_\gamma^{(m)}
:=
\{q_1+\cdots+q_{m-1}+r\mid r\in\gamma\}
\subset S^{(m)},
\]
\[
C_\gamma^{[m]}
:=
\{q_1+\cdots+q_{m-1}+r\mid r\in\gamma\}
\subset S^{[m]}.
\]
Here $C_\gamma^{[m]}$ denotes the closure of the locus obtained by letting $r$ vary in $\gamma_{\mathrm{sm}}\setminus\{q_1,\dots,q_{m-1}\}$.
At a node, the limit is the corresponding length-two subscheme with that tangent direction.

\begin{lem}\label{lem:nodal-boundary-intersection}
Let $\gamma\subset S$ be a reduced irreducible nodal curve and assume
$m>\delta(\gamma)$. 
If the distinct points $q_1,\dots,q_{m-1}\in\gamma$ contain all nodes of $\gamma$, then
\[
B^{[m]}\cdot C_\gamma^{[m]}
=
2\bigl(m-1+\delta(\gamma)\bigr).
\]
\end{lem}

\begin{proof}
Let $\nu:\widetilde{\gamma}\longrightarrow\gamma$ be the normalization.
The moving-point family defining $C_\gamma^{[m]}$ is parametrized by $\widetilde{\gamma}$.
The normalization contributes one preimage for each fixed smooth point and two for each node, giving
$m-1+\delta(\gamma)$ collisions. Each collision has multiplicity $2$ along the Hilbert--Chow boundary.
Thus the stated intersection number follows.
\end{proof}

Let $w\subset S$ be a smooth irreducible curve admitting a base-point-free $g^1_m$, equivalently a degree-$m$ morphism $f:w\longrightarrow\PP^1$.
The fibers of $f$ give a rational curve in $S^{[m]}$ with class $w_{[m]}$.
For every $L\in\Pic(S)$,
\[
L^{[m]}\cdot w_{[m]}=L\cdot w,
\qquad
B^{[m]}\cdot w_{[m]}=2g(w)-2+2m.
\]

Fix distinct points $p,q_1,\dots,q_{m-2}\in S$. 
Let $A^{[m]}$ be the curve obtained by varying a length-two subscheme $\xi$ supported at $p$ and setting $Z=\xi+q_1+\cdots+q_{m-2}$.
The intersection numbers \cite{RY18} are
\[
\begin{array}{c|cc}
 & L^{[m]} & B^{[m]}\\
\hline
C_\gamma^{[m]} & L\cdot\gamma & 2\bigl(m-1+\delta(\gamma)\bigr)\\
w_{[m]} & L\cdot w & 2\bigl(m-1+g(w)\bigr)\\
A^{[m]} & 0 & -2.
\end{array}
\]
Also, $L^{(m)}\cdot C_\gamma^{(m)}=L\cdot\gamma$.

\begin{lem}\label{lem:delpezzo-hilbert-mori}
Let $S$ be a del Pezzo surface with $2\le K_S^2\le7$. 
Then, for every $m\ge2$,
\[
\NEbar(S^{[m]})
=
\Cone\left\langle
\{C_E^{[m]}\}_E,\,
A^{[m]}
\right\rangle,
\]
where $E$ runs through the $(-1)$-curves on $S$.
\end{lem}

\begin{proof}
Bertram and Coskun determine $\Nef(S^{[m]})$ in \cite[Theorem~2.4]{BC13}. 
Writing a divisor as $L^{[m]}+\frac{c}{2}B^{[m]}$, the inequalities in
\cite[Theorem~2.4]{BC13} are
\[
c\le0,
\qquad
L\cdot E+(m-1)c\ge0
\]
for the $(-1)$-curves $E\subset S$.
By the intersection table,
$\left(L^{[m]}+\frac{c}{2}B^{[m]}\right)\cdot A^{[m]}=-c$
and
$\left(L^{[m]}+\frac{c}{2}B^{[m]}\right)\cdot C_E^{[m]}=L\cdot E+(m-1)c$.
Hence the facets of $\Nef(S^{[m]})$ are dual to $A^{[m]}$ and the $C_E^{[m]}$, which gives the stated description of $\NEbar(S^{[m]})$.
\end{proof}

\subsection{The nested Hilbert scheme $S^{[n,n+1]}$}
\label{subsec:nested-background}
Set $\Xcal:=S^{[n,n+1]}$.
The natural morphisms are
$\res:\Xcal\to S$, $p_b:\Xcal\to S^{[n]}$, and
$p_a:\Xcal\to S^{[n+1]}$,
where $\res(Z\subset Z')$ is the support of $Z'/Z$.
For $L\in\Pic(S)$, set
$L^{\diff}:=\res^*L$, $L^b:=p_b^*L^{[n]}$, and
$L^a:=p_a^*L^{[n+1]}$.
Similarly, set
$B^a:=p_a^*B^{[n+1]}$, $B^b:=p_b^*B^{[n]}$, and
$B^{\diff}:=B^a-B^b$.
\begin{lem}\label{lem:La=Ldiff+Lb-nested}
For every $L\in\Pic(S)$,
\[
L^a=L^{\diff}+L^b
\]
in $N^1(\Xcal)_\RR$.
\end{lem}

\begin{proof}
Let $\alpha:S\times S^{(n)}\to S^{(n+1)}$ be the addition map.
Since
\[
\HC_{n+1}\circ p_a=\alpha\circ(\res,p_b),
\]
it suffices to prove
\[
\alpha^*L^{(n+1)}=p_1^*L+p_2^*L^{(n)}.
\]
Both sides pull back to
$\sum_{i=1}^{n+1}\pr_i^*L$
under the finite quotient morphism
$S^{n+1}\to S\times S^{(n)}$.
Since pullback on $N^1$ is injective for finite surjective morphisms,
this identity follows, and hence so does the lemma.
\end{proof}

Let $q:=(\res,p_b)\colon \Xcal\longrightarrow S\times S^{[n]}$.
By \cite[Proposition~2.2]{ellingsrud1998intersection}, the morphism $q$ is the blow-up along the universal family $\mathcal Z_n$.
Let $\mathcal E$ be the exceptional divisor. 
Then $B^{\diff}=2\mathcal E$.
Since $q(S)=0$, Fogarty's description of the Picard group \cite{Fogarty1973AlgebraicFO}, together with the Picard group decomposition for the blow-up, gives, for a basis $H_1,\dots,H_\rho$ of
$N^1(S)_\RR$,
\[
N^1(\Xcal)_\RR
=
\bigoplus_{i=1}^{\rho}\RR H_i^{\diff}
\oplus
\bigoplus_{i=1}^{\rho}\RR H_i^b
\oplus
\RR B^b
\oplus
\RR B^{\diff}.
\]

We use the curve classes of \cite[Sections~4.2 and~4.3]{RY18}. 
Let $\gamma\subset S$ be a reduced irreducible nodal curve and assume $n>\delta(\gamma)$. 
Choose distinct points $q_1,\dots,q_n\in\gamma$ such that $q_1,\dots,q_{n-1}$ contain all nodes of $\gamma$ and $q_n\in\gamma_{\mathrm{sm}}$, and define
\[
C_\gamma^a
:=
\{(q_1+\cdots+q_n\subset q_1+\cdots+q_n+r)\mid r\in\gamma\},
\]
\[
C_\gamma^b
:=
\{(q_1+\cdots+q_{n-1}+r\subset q_1+\cdots+q_{n-1}+r+q_n)\mid r\in\gamma\}.
\]

Fix distinct points $p,q_2,\dots,q_n\in S$. 
Let $A_a$ be obtained by fixing $Z=p+q_2+\cdots+q_n$ and varying $Z'=\xi+q_2+\cdots+q_n$, where $\xi$ is a length-two subscheme supported at $p$. 
Let $\eta=\Spec(\Ocal_{S,p}/\mathfrak m_p^2)$, and let $A_b$ be obtained by fixing $Z'=\eta+q_2+\cdots+q_{n-1}$ and varying $Z=\xi+q_2+\cdots+q_{n-1}\subset Z'$, where $\xi\subset\eta$ has length two.
Using Lemma~\ref{lem:nodal-boundary-intersection} for $p_a$ and $p_b$, we get
\[
\begin{array}{c|cccc}
 & H^{\diff} & H^b & B^{\diff} & B^b\\
\hline
C_\gamma^a & H\cdot\gamma & 0 & 2(n+\delta(\gamma)) & 0\\
C_\gamma^b & 0 & H\cdot\gamma & 2 & 2(n-1+\delta(\gamma))\\
A_a & 0 & 0 & -2 & 0\\
A_b & 0 & 0 & 2 & -2.
\end{array}
\]

\subsection{The universal family $S^{[1,n]}$}
\label{subsec:universal-background}
Set
\[
\Xcal:=S^{[1,n]}
=
\{(p,Z)\in S\times S^{[n]}\mid p\in\Supp(Z)\}.
\]
Thus $\Xcal$ is the universal family over $S^{[n]}$; for $n>2$, this universal family
has rational, non-$\QQ$-Gorenstein singularities \cite[Theorem~1.1]{SONG2016348}
and hence is non-$\QQ$-factorial, and we write
$N^1(\Xcal)_\RR$ for Cartier divisors modulo numerical equivalence.

By \cite[Proposition~2.4]{RY18}, $\Pic(\Xcal)$ is generated by pullbacks along
$\pr_{\diff}:\Xcal\to S$ and $\pr_a:\Xcal\to S^{[n]}$.
We also consider $\pr_b:\Xcal\to S^{(n-1)}$, where
$\pr_b(p,Z)=\HC_n(Z)-p$.
For $L\in\Pic(S)$, set
$L^{\diff}:=\pr_{\diff}^*L$, $L^b:=\pr_b^*L^{(n-1)}$,
$L^a:=\pr_a^*L^{[n]}$, and $B^a:=\pr_a^*B^{[n]}$.

\begin{lem}\label{lem:La=Ldiff+Lb-universal}
For every $L\in\Pic(S)$,
\[
L^a=L^{\diff}+L^b
\]
in $N^1(\Xcal)_\RR$.
\end{lem}

\begin{proof}
The proof is the same as that of Lemma~\ref{lem:La=Ldiff+Lb-nested}.
\end{proof}

Since $q(S)=0$,
$N^1(S^{[n]})_\RR=N^1(S)_\RR\oplus\RR[B^{[n]}]$.
Thus, for a basis $H_1,\dots,H_\rho$ of $N^1(S)_\RR$, the classes $H_i^{\diff}$, $H_i^a$, and $B^a$ span $N^1(\Xcal)_\RR$.
By Lemma~\ref{lem:La=Ldiff+Lb-universal}, we may replace $H_i^a$ by $H_i^b$, and the intersection table below shows that these classes are
numerically independent. 
Hence
\[
N^1(\Xcal)_\RR
=
\bigoplus_{i=1}^{\rho}\RR H_i^{\diff}
\oplus
\bigoplus_{i=1}^{\rho}\RR H_i^b
\oplus
\RR B^a.
\]

We use the curve classes of \cite[Section~6.1]{RY18}.
For a nodal curve $\gamma\subset S$ with $n-1>\delta(\gamma)$, choose distinct points $q_1,\dots,q_{n-1}\in\gamma$ such that $q_1,\dots,q_{n-2}$ contain all nodes of $\gamma$ and $q_{n-1}\in\gamma_{\mathrm{sm}}$. 
Define
\[
C_\gamma^{\diff}
:=
\{(r,q_1+\cdots+q_{n-1}+r)\mid r\in\gamma\},
\]
\[
C_\gamma^a
:=
\{(q_{n-1},q_1+\cdots+q_{n-1}+r)\mid r\in\gamma\}.
\]
Fix a point $p\in S$ and distinct points $q_1,\dots,q_{n-2}$ away from $p$.
Let $A_a$ be obtained by varying the length-two subscheme $\xi$ supported at $p$ in $(p,\xi+q_1+\cdots+q_{n-2})$. 
Lemma~\ref{lem:nodal-boundary-intersection} gives
\[
\begin{array}{c|ccc}
 & H^{\diff} & H^b & B^a\\
\hline
C_\gamma^{\diff} & H\cdot\gamma & 0 & 2\bigl(n-1+\delta(\gamma)\bigr)\\
C_\gamma^a & 0 & H\cdot\gamma & 2\bigl(n-1+\delta(\gamma)\bigr)\\
A_a & 0 & 0 & -2.
\end{array}
\]

\section{The nested Hilbert scheme $S^{[n,n+1]}$}
\label{sec:nested}

\subsection{A nef cone decomposition}
\label{subsec:nested-decomposition}

Fix a basis $H_1,\dots,H_\rho$ of $N^1(S)_\RR$.
Every divisor class on $\Xcal=S^{[n,n+1]}$ admits a unique expression
\[
D=
\sum_{i=1}^{\rho}x_iH_i^{\diff}
+
\sum_{i=1}^{\rho}y_iH_i^b
+uB^{\diff}+vB^b.
\]
Set
\[
X:=\sum_{i=1}^{\rho}x_iH_i,
\qquad
Y:=\sum_{i=1}^{\rho}y_iH_i.
\]

\begin{prop}\label{prop:decomposition}
Let $\gamma_1,\dots,\gamma_k\subset S$ be rational nodal curves with $n>\delta(\gamma_i)$ for every $i$. 
Let $D=X^{\diff}+Y^b+uB^{\diff}+vB^b$
be nef on $\Xcal=S^{[n,n+1]}$. 
Assume
$\NEbar(S)=\Cone\langle\gamma_1,\dots,\gamma_k\rangle$
and
\[
\NEbar(S^{[m]})
=
\Cone\langle C_{\gamma_1}^{[m]},\dots,C_{\gamma_k}^{[m]},A^{[m]}\rangle
\]
for $m=n,n+1$. 
If there exists $M\in N^1(S)_\RR$ such that
\[
-2\bigl(n+\delta(\gamma_i)\bigr)u
\le
M\cdot\gamma_i
\le
\min\!\left(
X\cdot\gamma_i,
Y\cdot\gamma_i+2\bigl(n-1+\delta(\gamma_i)\bigr)(v-u)
\right)
\]
for every $i$, then
\[
D\in
\res^*\Nef(S)+p_b^*\Nef(S^{[n]})+p_a^*\Nef(S^{[n+1]}).
\]
\end{prop}

\begin{proof}
Set
\[
D_{\res}:=X-M,\qquad
D_a:=M^{[n+1]}+uB^{[n+1]},\qquad
D_b:=(Y-M)^{[n]}+(v-u)B^{[n]}.
\]
Then by Lemma~\ref{lem:La=Ldiff+Lb-nested},
$D=\res^*D_{\res}+p_a^*D_a+p_b^*D_b$.
The inequalities give
$D_{\res}\cdot\gamma_i\ge0$, $D_b\cdot C_{\gamma_i}^{[n]}\ge0$, and
$D_a\cdot C_{\gamma_i}^{[n+1]}\ge0$,
while
$D_a\cdot A^{[n+1]}=-2u=D\cdot A_a\ge0$.
Thus $D_{\res}$, $D_b$, and $D_a$ are nef.
\end{proof}
\begin{cor}\label{cor:decomposition}
Under the hypotheses of Proposition~\ref{prop:decomposition}, assume $k=\rho$.
Then
\[
\Nef(\Xcal)
=
\res^*\Nef(S)+p_b^*\Nef(S^{[n]})+p_a^*\Nef(S^{[n+1]}).
\]
\end{cor}
\begin{proof}
The classes $\gamma_1,\dots,\gamma_\rho$ form a basis of $N_1(S)_\RR$.
Hence the map
$N^1(S)_\RR\longrightarrow\RR^\rho$, $M\longmapsto
(M\cdot\gamma_1,\dots,M\cdot\gamma_\rho)$
is an isomorphism.
Thus the values $M\cdot\gamma_i$ can be prescribed independently.
The interval in Proposition~\ref{prop:decomposition} is nonempty exactly when
$D\cdot C_{\gamma_i}^a\ge0$ and $D\cdot C_{\gamma_i}^b\ge0$.
Since $D$ is nef, such an $M$ exists.
The reverse inclusion is immediate because pullbacks of nef divisors are nef.
\end{proof}
\begin{rmk}\label{rmk:dual-cone}
When $k=\rho$, the above argument is a numerical form of the dual-cone
computation used in \cite[Section~5.3]{RY18}. 
\end{rmk}

\begin{cor}\label{cor:nested-mori-by-duality}
Under the hypotheses of Corollary~\ref{cor:decomposition}, we have
\[
\NEbar(\Xcal)
=
\Cone\langle
C_{\gamma_1}^a,\dots,C_{\gamma_\rho}^a,
C_{\gamma_1}^b,\dots,C_{\gamma_\rho}^b,
A_a,A_b
\rangle.
\]
\end{cor}

\begin{proof}
Since $k=\rho$, the classes $\gamma_1,\dots,\gamma_\rho$ form a basis of $N_1(S)_\RR$.
The classes $C_{\gamma_1}^a,\dots,C_{\gamma_\rho}^a,
C_{\gamma_1}^b,\dots,C_{\gamma_\rho}^b,A_a,A_b$
form a basis of $N_1(\Xcal)_\RR$.
Their number is $2\rho+2=\dim N_1(\Xcal)_\RR$, and the intersection table shows that they are numerically independent.

Write
\[
\alpha
=
\sum_{i=1}^{\rho}a_iC_{\gamma_i}^a
+
\sum_{i=1}^{\rho}b_iC_{\gamma_i}^b
+
cA_a+dA_b.
\]
By Corollary~\ref{cor:decomposition}, $\alpha$ lies in the dual of
$\Nef(\Xcal)$ precisely when
\[
\res_*\alpha\in\NEbar(S),\qquad
(p_b)_*\alpha\in\NEbar(S^{[n]}),\qquad
(p_a)_*\alpha\in\NEbar(S^{[n+1]}).
\]
The pushforwards are
\[
\res_*\alpha=\sum_{i=1}^{\rho}a_i\gamma_i,\qquad
(p_b)_*\alpha=\sum_{i=1}^{\rho}b_iC_{\gamma_i}^{[n]}+dA^{[n]},
\]
and
\[
(p_a)_*\alpha
=
\sum_{i=1}^{\rho}(a_i+b_i)C_{\gamma_i}^{[n+1]}
+cA^{[n+1]}.
\]
Since the displayed classes generate the corresponding Mori cones by assumption, the first two conditions give
$a_i\ge0$, $b_i\ge0$, and $d\ge0$, while the third gives $c\ge0$.
Thus $\Nef(\Xcal)^\vee$ is the cone generated by the classes appearing in the statement.
By numerical duality, this is $\NEbar(\Xcal)$.
\end{proof}

\subsection{Examples}
\label{subsec:nested-examples}
\begin{exa}[The projective plane, \cite{RY18}]\label{ex:P2-nested}
For $S=\PP^2$ and $n\ge2$,
\[
\Nef(\Xcal)
=
\Cone\Bigl\langle
H^{\diff},
H^b,
H^b-\tfrac{1}{2(n-1)}B^b,
H^{\diff}+H^b-\tfrac{1}{2n}(B^{\diff}+B^b)
\Bigr\rangle.
\]
\end{exa}

\begin{exa}[Hirzebruch surfaces, \cite{RY18}]\label{ex:hirzebruch-nested}
Let $S=\mathbb F_e$, with minimal section $S_0$ and fiber $F$.
For $n\ge2$,
\[
\Nef(\Xcal)
=
\Cone\left\langle
\begin{aligned}
&F^{\diff},\ (S_0+eF)^{\diff},\ F^b,\ (S_0+eF)^b,\\
&(n-1)(S_0+(e+1)F)^b-\tfrac12 B^b,\\
&n(S_0+(e+1)F)^a-\tfrac12 B^a
\end{aligned}
\right\rangle.
\]
\end{exa}

\begin{exa}[A Picard rank one K3 surface, \cite{RY18}]\label{ex:K3}
Let $S$ be a general K3 surface with $\Pic(S)=\ZZ H$ and $H^2=2g-2$.
For $g\ge2$ and $n>g$, set
$D_m:=\frac{m-1+g}{2g-2}H^{[m]}-\frac12B^{[m]}$.
Then
\[
\Nef(\Xcal)
=
\Cone\langle H^{\diff},H^b,D_n^b,D_{n+1}^a\rangle.
\]
\end{exa}

\begin{exa}[Picard rank two Mori dream K3 surfaces, \cite{DER25}]
We consider Cases~I and III of \cite{DER25}.

\medskip

\noindent\emph{Case I.}
Assume
\[
\Pic(S)=\ZZ w_1\oplus\ZZ w_2,
\qquad
w_1^2=w_2^2=0,
\qquad
w_1\cdot w_2=k\ge2,
\]
where $w_1$ and $w_2$ are smooth elliptic curves.
Then
$\Nef(S)=\Cone\langle w_1,w_2\rangle$.
For $m\ge9k/8$, set
$D_m^{\mathrm I}:=\frac{m}{k}(w_1+w_2)^{[m]}-\frac12B^{[m]}$.
By \cite[Theorem~1.1]{DER25},
\[
\Nef(S^{[m]})
=
\Cone\left\langle
w_1^{[m]},w_2^{[m]},D_m^{\mathrm I}
\right\rangle.
\]
Here the classes $(w_i)_{[m]}$ are pencil curves rather than the moving-point curves $C_{w_i}^{[m]}$, so Corollary~\ref{cor:decomposition} does not apply directly.
The analogous dual-cone computation in \cite[Corollary~1.4]{DER25} gives
\[
\Nef(\Xcal)
=
\Cone\left\langle
w_1^{\diff},w_2^{\diff},
w_1^b,w_2^b,
(D_n^{\mathrm I})^b,(D_{n+1}^{\mathrm I})^a
\right\rangle
\]
for $n\ge9k/8$.

\medskip

\noindent\emph{Case III.}
Assume
\[
\Pic(S)=\ZZ w_1\oplus\ZZ w_2,
\qquad
w_1^2=w_2^2=-2,
\qquad
w_1\cdot w_2=k\ge3,
\]
where $w_1$ and $w_2$ are smooth rational curves.
Then $\NEbar(S)=\Cone\langle w_1,w_2\rangle$.
Assume
\[
n\ge k
\qquad\text{and}\qquad
8n^2+(2-9k)n+8\ge0.
\]
These inequalities also hold with $n+1$ in place of $n$.
By \cite[Theorem~1.1]{DER25}, for $m=n,n+1$,
\[
\NEbar(S^{[m]})
=
\Cone\left\langle
C_{w_1}^{[m]},C_{w_2}^{[m]},A^{[m]}
\right\rangle.
\]
Since $w_i\simeq\PP^1$, the pencil class $(w_i)_{[m]}$ of \cite{DER25} agrees numerically with $C_{w_i}^{[m]}$.

Set $H_1:=kw_1+2w_2$ and $H_2:=2w_1+kw_2$, and set
$D_m^{\mathrm{III}}
:=\frac{m-1}{k-2}(w_1+w_2)^{[m]}-\frac12B^{[m]}$.
Then
\[
\Nef(S^{[m]})
=
\Cone\left\langle
H_1^{[m]},H_2^{[m]},D_m^{\mathrm{III}}
\right\rangle.
\]
Since the hypotheses of Corollary~\ref{cor:decomposition} are satisfied,
we obtain
\[
\Nef(\Xcal)
=
\Cone\left\langle
H_1^{\diff},H_2^{\diff},
H_1^b,H_2^b,
(D_n^{\mathrm{III}})^b,(D_{n+1}^{\mathrm{III}})^a
\right\rangle.
\]
\end{exa}

\begin{cor}[Del Pezzo surfaces with $2\le K_S^2\le7$]\label{cor:delpezzo_nested}
Let $S$ be a del Pezzo surface with $2\le K_S^2\le7$ and let $n\ge2$.
Then
\[
\Nef(S^{[n,n+1]})
=
\res^*\Nef(S)+p_b^*\Nef(S^{[n]})+p_a^*\Nef(S^{[n+1]}).
\]
\end{cor}

\begin{proof}
Let $E$ run through the $(-1)$-curves on $S$.
By Lemma~\ref{lem:delpezzo-hilbert-mori},
\[
\NEbar(S^{[m]})
=
\Cone\langle\{C_E^{[m]}\}_E,A^{[m]}\rangle.
\]
For a nef divisor $D$, choose $M=-2nu(-K_S)$.
Since $-K_S\cdot E=1$ for every $(-1)$-curve $E$,
$M\cdot E=-2nu$, which is the lower endpoint in Proposition~\ref{prop:decomposition}.
The remaining inequalities,
$M\cdot E\le X\cdot E$ and
$M\cdot E\le Y\cdot E+2(n-1)(v-u)$,
are equivalent to
$D\cdot C_E^a\ge0$ and $D\cdot C_E^b\ge0$,
so they hold since $D$ is nef.
Thus Proposition~\ref{prop:decomposition} shows that every nef divisor lies in the stated sum.
The opposite inclusion is immediate because pullbacks of nef divisors are nef.
\end{proof}

\section{The universal family $S^{[1,n]}$}
\label{sec:universal}

\subsection{A nef cone decomposition}
\label{subsec:universal-decomposition}

Fix a basis $H_1,\dots,H_\rho$ of $N^1(S)_\RR$.
Every divisor class on $\Xcal=S^{[1,n]}$ admits a unique expression
\[
D=
\sum_{i=1}^{\rho}x_iH_i^{\diff}
+
\sum_{i=1}^{\rho}y_iH_i^b
+uB^a.
\]
Set
\[
X:=\sum_{i=1}^{\rho}x_iH_i,
\qquad
Y:=\sum_{i=1}^{\rho}y_iH_i.
\]

The assumption
$\NEbar(S)=\Cone\langle\gamma_1,\dots,\gamma_k\rangle$
also gives
\[
\NEbar(S^{(n-1)})
=
\Cone\langle C_{\gamma_1}^{(n-1)},\dots,C_{\gamma_k}^{(n-1)}\rangle.
\]
This follows from the quotient morphism $S^{n-1}\to S^{(n-1)}$.

\begin{prop}\label{prop:universal-decomposition}
Let $\gamma_1,\dots,\gamma_k\subset S$ be rational nodal curves with $n-1>\delta(\gamma_i)$ for every $i$.
Let $D=X^{\diff}+Y^b+uB^a$ be nef on $\Xcal=S^{[1,n]}$.
Assume
$\NEbar(S)=\Cone\langle\gamma_1,\dots,\gamma_k\rangle$
and
\[
\NEbar(S^{[n]})
=
\Cone\langle C_{\gamma_1}^{[n]},\dots,C_{\gamma_k}^{[n]},A^{[n]}\rangle.
\]
If there exists $M\in N^1(S)_\RR$ such that
\[
-2\bigl(n-1+\delta(\gamma_i)\bigr)u
\le
M\cdot\gamma_i
\le
\min\bigl(X\cdot\gamma_i,Y\cdot\gamma_i\bigr)
\]
for every $i$, then
\[
D\in
\pr_{\diff}^*\Nef(S)
+\pr_b^*\Nef(S^{(n-1)})
+\pr_a^*\Nef(S^{[n]}).
\]
\end{prop}
\begin{proof}
Set 
\[
D_{\diff}:=X-M,\qquad
D_b:=(Y-M)^{(n-1)},\qquad
D_a:=M^{[n]}+uB^{[n]}.
\]
By Lemma~\ref{lem:La=Ldiff+Lb-universal},
$D=\pr_{\diff}^*D_{\diff}+\pr_b^*D_b+\pr_a^*D_a$.
The assumed inequalities show that $D_{\diff}$, $D_b$, and $D_a$ have nonnegative intersection with the non-punctual generators of their respective Mori cones.
By the intersection table and the nefness of $D$,
$D_a\cdot A^{[n]}=-2u=D\cdot A_a\ge0$.
Since these classes generate the corresponding Mori cones, $D_{\diff}$, $D_b$, and $D_a$ are nef.
\end{proof}

\begin{cor}\label{cor:universal-nef-decomposition}
Under the hypotheses of Proposition~\ref{prop:universal-decomposition}, assume $k=\rho$. 
Then
\[
\Nef(\Xcal)
=
\pr_{\diff}^*\Nef(S)
+\pr_b^*\Nef(S^{(n-1)})
+\pr_a^*\Nef(S^{[n]}).
\]
\end{cor}

\begin{proof}
The argument of Corollary~\ref{cor:decomposition} applies.
\end{proof}

\begin{cor}\label{cor:universal-mori-by-duality}
Under the hypotheses of Corollary~\ref{cor:universal-nef-decomposition}, we have
\[
\NEbar(\Xcal)
=
\Cone\langle
C_{\gamma_1}^{\diff},\dots,C_{\gamma_\rho}^{\diff},
C_{\gamma_1}^a,\dots,C_{\gamma_\rho}^a,
A_a
\rangle.
\]
\end{cor}

\begin{proof}
Since $k=\rho$, the classes $\gamma_1,\dots,\gamma_\rho$ form a basis of $N_1(S)_\RR$.
By the intersection table, the classes
$C_{\gamma_1}^{\diff},\dots,C_{\gamma_\rho}^{\diff},
C_{\gamma_1}^a,\dots,C_{\gamma_\rho}^a,A_a$
are numerically independent.
Since $\dim N_1(\Xcal)_\RR=2\rho+1$, they form a basis of $N_1(\Xcal)_\RR$.

Write
\[
\alpha
=
\sum_{i=1}^{\rho}a_iC_{\gamma_i}^{\diff}
+
\sum_{i=1}^{\rho}b_iC_{\gamma_i}^a
+
cA_a.
\]
By Corollary~\ref{cor:universal-nef-decomposition}, $\alpha$ lies in $\Nef(\Xcal)^\vee$ precisely when
\[
(\pr_{\diff})_*\alpha\in\NEbar(S),
\qquad
(\pr_b)_*\alpha\in\NEbar(S^{(n-1)}),
\qquad
(\pr_a)_*\alpha\in\NEbar(S^{[n]}).
\]
The pushforwards are
\[
(\pr_{\diff})_*\alpha=\sum_i a_i\gamma_i,\qquad
(\pr_b)_*\alpha=\sum_i b_iC_{\gamma_i}^{(n-1)},
\]
and
\[
(\pr_a)_*\alpha
=
\sum_i(a_i+b_i)C_{\gamma_i}^{[n]}+cA^{[n]}.
\]
Hence the first two conditions give $a_i,b_i\ge0$, and the third gives $c\ge0$.
Thus $\NEbar(\Xcal)$ is generated by
$C_{\gamma_i}^{\diff}$, $C_{\gamma_i}^a$, and $A_a$.
\end{proof}

\subsection{Examples}
\label{subsec:universal-examples}

\begin{exa}[The projective plane, \cite{RY18}]\label{ex:P2-univ}
For $S=\PP^2$,
\[
\Nef(\Xcal)
=
\Cone\Bigl\langle
H^{\diff},
H^b,
H^{\diff}+H^b-\tfrac{1}{2(n-1)}B^a
\Bigr\rangle.
\]
\end{exa}

\begin{exa}[Hirzebruch surfaces, \cite{RY18}]\label{ex:hirzebruch-univ}
For $S=\mathbb F_e$,
\[
\Nef(\Xcal)
=
\Cone\left\langle
\begin{aligned}
&F^{\diff},\ (S_0+eF)^{\diff},\ F^b,\ (S_0+eF)^b,\\
&(n-1)(S_0+(e+1)F)^a-\tfrac12 B^a
\end{aligned}
\right\rangle.
\]
\end{exa}

\begin{exa}[A Picard rank one K3 surface, \cite{RY18}]
Let $S$ be a general K3 surface with $\Pic(S)=\ZZ H$ and $H^2=2g-2$.
For $g\ge2$ and $n>g$, set
$D_n:=\frac{n-1+g}{2g-2}H^{[n]}-\frac12B^{[n]}$.
Then
\[
\Nef(\Xcal)=\Cone\langle H^{\diff},H^b,D_n^a\rangle.
\]
\end{exa}

\begin{cor}[Del Pezzo surfaces with $2\le K_S^2\le7$]\label{cor:delpezzo_univ}
Let $S$ be a del Pezzo surface with $2\le K_S^2\le7$ and let $n\ge2$.
Then
\[
\Nef(S^{[1,n]})
=
\pr_{\diff}^*\Nef(S)
+\pr_b^*\Nef(S^{(n-1)})
+\pr_a^*\Nef(S^{[n]}).
\]
\end{cor}

\begin{proof}
Let $E$ run through the $(-1)$-curves on $S$.
By Lemma~\ref{lem:delpezzo-hilbert-mori},
\[
\NEbar(S^{[n]})
=
\Cone\langle\{C_E^{[n]}\}_E,A^{[n]}\rangle.
\]
For a nef divisor $D=X^{\diff}+Y^b+uB^a$, take
$M=-2(n-1)u(-K_S)$.
Since $-K_S\cdot E=1$ for every $(-1)$-curve $E$, the lower bound in
Proposition~\ref{prop:universal-decomposition} is attained.
The upper bounds are equivalent to
$D\cdot C_E^{\diff}\ge0$ and $D\cdot C_E^a\ge0$,
and hence hold because $D$ is nef.
Thus Proposition~\ref{prop:universal-decomposition} shows that every nef divisor lies in the stated sum.
The opposite inclusion is immediate because pullbacks of nef divisors are nef.
\end{proof}

\section{The degree-one del Pezzo case}
\label{sec:degree-one}

Let $S$ be a del Pezzo surface of degree $1$, and set $F:=-K_S$.
For $m\ge2$, let $F_{[m]}$ denote the curve class obtained by letting $m$ points move in a $g^1_m$ on a smooth member of $|F|$.
Its intersections with the standard divisor classes are
\[
L^{[m]}\cdot F_{[m]}=L\cdot F,
\qquad
B^{[m]}\cdot F_{[m]}=2m.
\]
For a $(-1)$-curve $E\subset S$, the moving-point class $C_E^{[m]}$ defined above agrees numerically with the class of a line in $E^{[m]}\simeq\PP^m$, so $B^{[m]}\cdot C_E^{[m]}=2m-2$.
The Mori cone of $S^{[m]}$ is
\[
\NEbar(S^{[m]})
=
\Cone\left\langle
\{C_E^{[m]}\}_E,\ F_{[m]},\ A^{[m]}
\right\rangle,
\]
where $E$ runs through the $(-1)$-curves on $S$ \cite[Theorem~5.1]{BHL+16}; see also the correction \cite{BHL+16erratum}.

Fix a $(-1)$-curve $E$ and set $D:=2F-E$. 
By \cite[Section~5.1]{BHL+16}, $D$ is again a $(-1)$-curve; hence $\delta(D)=0$, $E+D=2F$, and $E\cdot D=3$.

\begin{lem}
\label{lem:degree-one-hilbert-relation}
For every $m\ge2$,
\[
C_E^{[m]}+C_D^{[m]}
=
2F_{[m]}+2A^{[m]}
\]
in $N_1(S^{[m]})_\RR$.
\end{lem}

\begin{proof}
For every $L\in N^1(S)_\RR$,
\[
L^{[m]}\cdot\bigl(C_E^{[m]}+C_D^{[m]}\bigr)
=2L\cdot F
=L^{[m]}\cdot\bigl(2F_{[m]}+2A^{[m]}\bigr).
\]
Moreover,
\[
B^{[m]}\cdot\bigl(C_E^{[m]}+C_D^{[m]}\bigr)
=4m-4
=B^{[m]}\cdot\bigl(2F_{[m]}+2A^{[m]}\bigr).
\]
Here we use
$B^{[m]}\cdot C_E^{[m]}=B^{[m]}\cdot C_D^{[m]}=2m-2$,
$B^{[m]}\cdot F_{[m]}=2m$, and $B^{[m]}\cdot A^{[m]}=-2$.
Since $L^{[m]}$ together with $B^{[m]}$ span $N^1(S^{[m]})_\RR$,
the claim follows from the perfect numerical pairing.
\end{proof}

Set $\Xcal:=S^{[n,n+1]}$ and
\[
P:=\res^*\Nef(S)+p_b^*\Nef(S^{[n]})+p_a^*\Nef(S^{[n+1]}).
\]
By the projection formula and numerical duality,
\[
P^\vee
=
\res_*^{-1}\bigl(\NEbar(S)\bigr)
\cap
(p_b)_*^{-1}\bigl(\NEbar(S^{[n]})\bigr)
\cap
(p_a)_*^{-1}\bigl(\NEbar(S^{[n+1]})\bigr)
=:K.
\]
Since $P\subseteq\Nef(\Xcal)$, we have $\NEbar(\Xcal)\subseteq K$.
It suffices to exhibit a class in $K\setminus\NEbar(\Xcal)$.

Recall that $q=(\res,p_b)$ is the blow-up of $S\times S^{[n]}$ along the universal family $\mathcal Z_n$.
Let $\mathcal E$ denote the exceptional divisor. Then
$B^{\diff}=p_a^*B^{[n+1]}-p_b^*B^{[n]}=2\mathcal E$.

Let $C_E^a$ and $C_D^b$ be the nested curve classes defined in Subsection~\ref{subsec:nested-background}.
Their pushforwards are
\[
\begin{array}{c|ccc}
& \res_* & (p_b)_* & (p_a)_*\\
\hline
C_E^a & E & 0 & C_E^{[n+1]}\\
C_D^b & 0 & C_D^{[n]} & C_D^{[n+1]}\\
A_a & 0 & 0 & A^{[n+1]}.
\end{array}
\]
Set $T:=C_E^a+C_D^b-2A_a$.
By Lemma~\ref{lem:degree-one-hilbert-relation},
\[
\res_*T=E,\qquad
(p_b)_*T=C_D^{[n]},\qquad
(p_a)_*T=2F_{[n+1]}.
\]
Hence $T\in K$.

\begin{lem}
\label{lem:degree-one-collision-bound}
Let $\Gamma\subset\Xcal$ be an irreducible curve such that
$\res_*(\Gamma)=aE$ and $(p_b)_*(\Gamma)=aC_D^{[n]}$
for some $a>0$.
Then
\[
\mathcal E\cdot\Gamma\le(n+2)a.
\]
\end{lem}

\begin{proof}
Let $\nu\colon C\to\Gamma$ be the normalization and set $f=\res\circ\nu$, $g=p_b\circ\nu$.
Then $\deg f=a$.

Let $S^{[1,\ldots,n]}$ denote the full flag Hilbert scheme parametrizing chains
$Z_1\subset Z_2\subset\cdots\subset Z_n$ with $\length(Z_i)=i$.
The forgetful morphism
$S^{[1,\ldots,n]}\longrightarrow S^{[n]}$,
$(Z_1\subset\cdots\subset Z_n)\longmapsto Z_n$,
is projective and surjective.
Hence some irreducible component of
$C\times_{S^{[n]}}S^{[1,\ldots,n]}$
dominates $C$.
Since it is projective over $C$, choose an irreducible curve dominating $C$.
After normalizing, we obtain a finite surjective morphism $\pi\colon C'\to C$ of degree $r$, together with a lift of $g\circ\pi$ to the full flag Hilbert scheme.
Pulling back the universal flag on $S^{[1,\ldots,n]}$ along the lift $C'\to S^{[1,\ldots,n]}$, we obtain
\[
\mathcal Z_1
\subset
\mathcal Z_2
\subset
\cdots
\subset
\mathcal Z_n
\subset
S\times C',
\]
where $\mathcal Z_i$ is finite flat of degree $i$ over $C'$.
For each $i$, the exact sequence
\[
0\longrightarrow
\mathcal I_{\mathcal Z_{i-1}}/\mathcal I_{\mathcal Z_i}
\longrightarrow
\mathcal O_{\mathcal Z_i}
\longrightarrow
\mathcal O_{\mathcal Z_{i-1}}
\longrightarrow 0
\]
shows that $\mathcal I_{\mathcal Z_{i-1}}/\mathcal I_{\mathcal Z_i}$ is finite flat of rank one over $C'$.
Its support therefore has degree one over $C'$ and defines a section $z_i\colon C'\to S$.
The Hilbert--Chow cycle of $\mathcal Z_n$ is then $z_1+\cdots+z_n$.
Set $f':=f\circ\pi$, so that $\deg f'=ar$.
For every $H\in N^1(S)_\RR$,
\[
H\cdot\sum_{i=1}^n(z_i)_*[C']
=
r\,H^{[n]}\cdot(p_b)_*(\Gamma)
=
ar\,H\cdot D.
\]
Hence $\sum_{i=1}^n(z_i)_*[C']=arD$ in $N_1(S)_\RR$.

Since $\RR_{\ge0}[D]$ is an extremal ray of $\NEbar(S)$, every nonzero class $(z_i)_*[C']$ lies on this ray. 
Hence, for each nonconstant $z_i$, $(z_i)_*[C']=e_iD$ for some $e_i>0$. 
Since $D^2=-1$, the image of $z_i$ is necessarily $D$: otherwise it would be an irreducible curve distinct from $D$ having negative intersection with $D$. 
Then
\[
\sum_{z_i\text{ nonconstant}}e_i=ar.
\] 

The curve $\Gamma$ is not contained in $\mathcal E$.
Otherwise its image under $q$ would lie in the universal family $\mathcal Z_n$.
After base change to $C'$, we would have
\[
f'(t)\in\Supp(\mathcal Z_{n,t})=\{z_1(t),\ldots,z_n(t)\}
\]
for every $t\in C'$. Thus
\[
C'=\bigcup_{i=1}^n Z_i,
\qquad
Z_i:=\{t\in C'\mid f'(t)=z_i(t)\}.
\]
Each $Z_i$ is closed, since it is the inverse image of the diagonal under
$(f',z_i):C'\longrightarrow S\times S$.
Since $C'$ is irreducible, $Z_i=C'$ for some $i$.
Hence $f'=z_i$ as morphisms, and therefore
\[
z_i(C')=f'(C')=E.
\]
Every nonconstant $z_i$ has image $D$, whereas a constant $z_i$ cannot equal the nonconstant map $f'$.
This is a contradiction.

Let $\Gamma_{z_i}\subset S\times C'$ denote the graph of $z_i$.
Since $\mathcal I_{\mathcal Z_{i-1}}/\mathcal I_{\mathcal Z_i}$ is supported on $\Gamma_{z_i}$, its restriction to the complement of $\Gamma_{z_i}$ vanishes.
Equivalently,
$\mathcal I_{\mathcal Z_{i-1}}\mathcal I_{\Gamma_{z_i}}
\subseteq\mathcal I_{\mathcal Z_i}$.
Applying this inclusion successively for $i=1,\dots,n$ gives
\[
\prod_{i=1}^n\mathcal I_{\Gamma_{z_i}}
\subseteq
\mathcal I_{\mathcal Z_n}.
\]

Let $\Gamma_{f'}\subset S\times C'$ be the graph of $f'$.
Since $\Gamma\not\subset\mathcal E$, the pullback of $\mathcal I_{\mathcal Z_n}$ to $\Gamma_{f'}\simeq C'$ is nonzero.
For $t\in C'$, set
\[
m_t:=\ord_t(\Gamma_{f'}^*\mathcal I_{\mathcal Z_n}),
\qquad
\iota_t(f',z_i):=
\ord_t(\Gamma_{f'}^*\mathcal I_{\Gamma_{z_i}}).
\]

Pulling back
$\prod_{i=1}^n\mathcal I_{\Gamma_{z_i}}\subseteq\mathcal I_{\mathcal Z_n}$
to $\Gamma_{f'}$ and taking orders at $t$ gives
\[
m_t\le\sum_{i=1}^n\iota_t(f',z_i).
\]
Since $q$ is the blow-up along the universal family
$\mathcal Z_n\subset S\times S^{[n]}$,
\[
\mathcal O_{\Xcal}(-\mathcal E)
=
\mathcal I_{\mathcal Z_n}\mathcal O_{\Xcal}.
\]
Let $h\colon C'\longrightarrow\Xcal$ be the composition
$C'\xrightarrow{\pi}C\to\Gamma\hookrightarrow\Xcal$.
Since $q\circ h=(f',g\circ\pi)$, the pullback of $\mathcal I_{\mathcal Z_n}$ to $C'$ is $\Gamma_{f'}^*\mathcal I_{\mathcal Z_n}$.
Hence
\[
h^*\mathcal O_{\Xcal}(-\mathcal E)
=
\Gamma_{f'}^*\mathcal I_{\mathcal Z_n}
=
\mathcal O_{C'}\left(-\sum_{t\in C'}m_t[t]\right).
\]
Taking degrees and using $h_*[C']=r[\Gamma]$, where $r=\deg(\pi)$, gives
\[
\sum_{t\in C'}m_t
=
r\,\mathcal E\cdot\Gamma.
\]
Thus
\[
r\,\mathcal E\cdot\Gamma
\le
\sum_{i=1}^n\sum_{t\in C'}\iota_t(f',z_i).
\]

Suppose first that $z_i$ is nonconstant.
Then
\[
z_i(C')=D,
\qquad
(z_i)_*[C']=e_iD.
\]
If $f'(t)=z_i(t)$, then the common value lies in $E\cap D$.
Choose local coordinates $(x,y)$ on $S$ near this point such that $E=\{x=0\}$.
Since $z_i(C')=D$ and $D\neq E$, the function $x\circ z_i$ is not identically zero.
The coincidence of the two graphs at $t$ implies
\[
\iota_t(f',z_i)
\le
\ord_t(x\circ z_i)
=
\ord_t(z_i^*s_E),
\]
where $s_E$ is a local equation of $E$.
Therefore
\[
\sum_{t\in C'}\iota_t(f',z_i)
\le
\deg z_i^*E
=
E\cdot(z_i)_*[C']
=
3e_i.
\]
Since $\sum_{z_i\text{ nonconstant}}e_i=ar$, the total contribution from the nonconstant $z_i$ is at most $3ar$.

Now suppose that $z_i$ is constant, with value $q_i\in S$.
If $q_i\notin E$, then the graphs of $f'$ and $z_i$ are disjoint.
If $q_i\in E$, then
\[
\sum_{t\in C'}\iota_t(f',z_i)
\le
\deg (f')^*\mathcal O_E(q_i)
=
ar.
\]
There are at most $n-1$ constant $z_i$, so their total contribution is at most $(n-1)ar$.
Therefore
\[
r\,\mathcal E\cdot\Gamma
\le
3ar+(n-1)ar
=
(n+2)ar.
\]
Dividing by $r$ gives
\[
\mathcal E\cdot\Gamma\le(n+2)a.
\]
\end{proof}

\begin{prop}
\label{prop:degree-one-supporting-ray}
There exists a semiample $\QQ$-divisor $L$ on $\Xcal$ such that
\[
K\cap L^\perp
=
\RR_{\ge0}[T].
\]
\end{prop}

\begin{proof}
Since $\NEbar(S)$ is rational polyhedral, choose $N_E,N_D\in\Nef(S)_\QQ$ such that
\[
\NEbar(S)\cap N_E^\perp=\RR_{\ge0}[E],
\qquad
\NEbar(S)\cap N_D^\perp=\RR_{\ge0}[D].
\]
Set $d:=N_D\cdot F>0$ and choose $\lambda\in\QQ$ with $\lambda>1/d$.
Define
\[
Q_D:=\bigl((n-1)F+\lambda N_D\bigr)^{[n]}-\frac12B^{[n]},
\qquad
Q_F:=(n+1)F^{[n+1]}-\frac12B^{[n+1]}.
\]
For every $(-1)$-curve $E'\neq D$,
$Q_D\cdot C_{E'}^{[n]}=\lambda N_D\cdot E'>0$,
while
\[
Q_D\cdot C_D^{[n]}=0,
\qquad
Q_D\cdot F_{[n]}=\lambda d-1>0,
\qquad
Q_D\cdot A^{[n]}=1.
\]
Thus
\[
\NEbar(S^{[n]})\cap Q_D^\perp
=
\RR_{\ge0}[C_D^{[n]}].
\]
For $Q_F$,
\[
Q_F\cdot C_{E'}^{[n+1]}=1,
\qquad
Q_F\cdot F_{[n+1]}=0,
\qquad
Q_F\cdot A^{[n+1]}=1,
\]
so
\[
\NEbar(S^{[n+1]})\cap Q_F^\perp
=
\RR_{\ge0}[F_{[n+1]}].
\]

Since $N_E-K_S=N_E+F$ is ample, $N_E$ is semiample by the base-point-free theorem \cite[Theorem~3.3]{kollar1998birational}.
Also $K_{S^{[m]}}=(K_S)^{[m]}$, and
$Q_D-K_{S^{[n]}}=\bigl(nF+\lambda N_D\bigr)^{[n]}-\frac12B^{[n]}$.
Its intersections with the generators of $\NEbar(S^{[n]})$ are
\[
\begin{aligned}
(Q_D-K_{S^{[n]}})\cdot C_D^{[n]}&=1,\\
(Q_D-K_{S^{[n]}})\cdot C_{E'}^{[n]}
&=1+\lambda N_D\cdot E'>0
\qquad(E'\neq D),\\
(Q_D-K_{S^{[n]}})\cdot F_{[n]}&=\lambda d>0,\\
(Q_D-K_{S^{[n]}})\cdot A^{[n]}&=1.
\end{aligned}
\]
Thus $Q_D$ is nef, and $Q_D-K_{S^{[n]}}$ is ample.
Hence $Q_D$ is semiample by the base-point-free theorem.
Similarly,
\[
Q_F-K_{S^{[n+1]}}
=
(n+2)F^{[n+1]}
-\frac12B^{[n+1]}
\]
has positive intersection with every extremal ray of $\NEbar(S^{[n+1]})$, so $Q_F$ is nef and $Q_F-K_{S^{[n+1]}}$ is ample.
Thus $Q_F$ is semiample by the base-point-free theorem.

Set $L:=\res^*N_E+p_b^*Q_D+p_a^*Q_F$.
Then $L$ is semiample and $L\cdot T=0$.

Let $\gamma\in K\cap L^\perp$. Its three pushforwards lie in the corresponding Mori cones, and
\[
0
=
N_E\cdot\res_*\gamma
+
Q_D\cdot(p_b)_*\gamma
+
Q_F\cdot(p_a)_*\gamma.
\]
Since all three terms are nonnegative, they all vanish.
Hence
\[
\res_*\gamma=aE,
\qquad
(p_b)_*\gamma=bC_D^{[n]},
\qquad
(p_a)_*\gamma=cF_{[n+1]}
\]
for some $a,b,c\ge0$.
The relation $H^a=H^{\diff}+H^b$ gives
$c(H\cdot F)=a(H\cdot E)+b(H\cdot D)$
for every $H\in N^1(S)_\RR$, so $cF=aE+bD$.
Since $D=2F-E$, it follows that $a=b$ and $c=2a$.

The blow-up decomposition
\[
N^1(\Xcal)_\RR
=
q^*N^1(S\times S^{[n]})_\RR
\oplus
\RR[\mathcal E]
\]
shows that the relative Picard number of $q$ is one.
Therefore
\[
\ker\left(
q_*\colon N_1(\Xcal)_\RR
\longrightarrow
N_1(S\times S^{[n]})_\RR
\right)
=
\RR[A_a].
\]
Hence $\gamma-a(C_E^a+C_D^b)=uA_a$ for some $u\in\RR$.
Pushing forward by $p_a$ and using Lemma~\ref{lem:degree-one-hilbert-relation}, we obtain
\[
\begin{aligned}
2aF_{[n+1]}
&=
a\bigl(C_E^{[n+1]}+C_D^{[n+1]}\bigr)
+
uA^{[n+1]}\\
&=
2aF_{[n+1]}
+
(2a+u)A^{[n+1]}.
\end{aligned}
\]
Hence $u=-2a$ and $\gamma=aT$, proving $K\cap L^\perp=\RR_{\ge0}[T]$.
\end{proof}

\begin{thm}
\label{thm:degree-one-failure}
Let $S$ be a del Pezzo surface of degree $1$ and let $n\ge2$.
Then
\[
\res^*\Nef(S)
+
p_b^*\Nef(S^{[n]})
+
p_a^*\Nef(S^{[n+1]})
\subsetneq
\Nef(S^{[n,n+1]}).
\]
\end{thm}

\begin{proof}
Let $L$ be as in Proposition~\ref{prop:degree-one-supporting-ray}.
Suppose that $T\in\NEbar(\Xcal)$. Since $L$ is semiample, a sufficiently divisible multiple defines a morphism $\varphi\colon\Xcal\to Y$ with $rL=\varphi^*A$ for an ample divisor $A$ on $Y$.

As $T\neq0$ and $L\cdot T=0$, $\varphi$ is not finite.
Choose an irreducible curve $\Gamma$ in a positive-dimensional fiber.
Then $[\Gamma]\in K\cap L^\perp$, so Proposition~\ref{prop:degree-one-supporting-ray}
gives $[\Gamma]=aT$ for some $a>0$.
Thus $\res_*\Gamma=aE$ and $(p_b)_*\Gamma=aC_D^{[n]}$,
and Lemma~\ref{lem:degree-one-collision-bound} yields
$\mathcal E\cdot\Gamma\le(n+2)a$.

Using $2\mathcal E=B^a-B^b$ and the pushforwards of $T$ computed above,
\[
\mathcal E\cdot\Gamma
=
a\,\mathcal E\cdot T
=
\frac a2
\left(
B^{[n+1]}\cdot 2F_{[n+1]}
-
B^{[n]}\cdot C_D^{[n]}
\right)
=
\frac a2\left(4(n+1)-(2n-2)\right)
=
a(n+3).
\]
This is a contradiction.
Hence $T\in K\setminus\NEbar(\Xcal)$.
Since $K=P^\vee$, this gives $P\subsetneq\Nef(\Xcal)$.
\end{proof}

\begin{rmk}
The additional Mori ray $F_{[n+1]}$ is contracted by the semiample divisor
\[
Q_F=(n+1)F^{[n+1]}-\frac12B^{[n+1]}.
\]
The associated contraction
\[
\phi_F\colon S^{[n+1]}\to
Y_F:=
\operatorname{Proj}
\bigoplus_{r\ge0}
H^0\!\left(S^{[n+1]},rQ_F\right)
\]
satisfies
\[
\NEbar(S^{[n+1]})\cap Q_F^\perp
=
\RR_{\ge0}[F_{[n+1]}].
\]
\end{rmk}

\section{The degree-one universal family case}
\label{sec:degree-one-universal}

Let $S$ be a del Pezzo surface of degree $1$.
We retain the notation $F=-K_S$ and the $(-1)$-curves $E,D$ from Section~\ref{sec:degree-one}, so that $E+D=2F$ and $E\cdot D=3$.

We use the following variant of Lemma~\ref{lem:degree-one-collision-bound}, in which only the Hilbert--Chow pushforward is prescribed.

\begin{lem}
\label{lem:degree-one-collision-bound-hc}
Let $m\ge1$, and let
\[
q_m=(\res,p_b):
S^{[m,m+1]}\longrightarrow S\times S^{[m]}
\]
be the blow-up along the universal family. 
Let $\mathcal E_m$ denote its exceptional divisor. 
Suppose that $\Gamma\subset S^{[m,m+1]}$ is an irreducible curve such that
\[
\res_*(\Gamma)=aE,
\qquad
(\HC_m\circ p_b)_*(\Gamma)=aC_D^{(m)}
\]
for some $a>0$. 
For $m=1$, we identify $S^{(1)}$ with $S$ and $C_D^{(1)}$ with $D$. 
Then
\[
\mathcal E_m\cdot\Gamma\le(m+2)a.
\]
\end{lem}

\begin{proof}
For $m\ge2$, pass to a finite cover $\pi:C'\to C$ of degree $r$ and lift to the full flag Hilbert scheme.
The hypothesis gives
\[
\sum_{i=1}^m(z_i)_*[C']=arD.
\]
For every $H\in N^1(S)_\RR$,
\[
H\cdot\sum_{i=1}^m(z_i)_*[C']
=
rH^{(m)}\cdot(\HC_m\circ p_b)_*(\Gamma)
=
arH\cdot D.
\]
Hence $\sum_{i=1}^m(z_i)_*[C']=arD$ in $N_1(S)_\RR$.
The same argument gives
\[
r\mathcal E_m\cdot\Gamma\le3ar+(m-1)ar=(m+2)ar.
\]

For $m=1$, $q_1:S^{[1,2]}\to S\times S$ is the blow-up along the diagonal.
Let $\nu:C\to\Gamma$ be the normalization and set $f=\res\circ\nu$, $g=p_b\circ\nu$.
Then $f_*[C]=aE$ and $g_*[C]=aD$.
Since $E\ne D$, $\Gamma$ is not contained in $\mathcal E_1$.
If $\iota_t(f,g)$ denotes the local coincidence multiplicity, then, using a local equation $s_E$ of $E$ at a common value,
\[
\iota_t(f,g)\le\ord_t(g^*s_E).
\]
Thus
\[
\mathcal E_1\cdot\Gamma
=\sum_t\iota_t(f,g)
\le E\cdot g_*[C]
=a(E\cdot D)=3a.
\]
\end{proof}

Set $\Xcal:=S^{[1,n]}$, $n\ge2$, and define
\[
P_{\mathrm{univ}}
:=
\pr_{\diff}^*\Nef(S)
+\pr_b^*\Nef(S^{(n-1)})
+\pr_a^*\Nef(S^{[n]}).
\]
By the projection formula and numerical duality,
\[
P_{\mathrm{univ}}^\vee
=(\pr_{\diff})_*^{-1}(\NEbar(S))
\cap(\pr_b)_*^{-1}(\NEbar(S^{(n-1)}))
\cap(\pr_a)_*^{-1}(\NEbar(S^{[n]}))
=:K_{\mathrm{univ}}.
\]

For $n=2$, set $S^{(1)}=S$ and $C_D^{(1)}=D$. 
The pushforwards are
\[
\begin{array}{c|ccc}
& (\pr_{\diff})_* & (\pr_b)_* & (\pr_a)_*\\
\hline
C_E^{\diff}&E&0&C_E^{[n]}\\
C_D^a&0&C_D^{(n-1)}&C_D^{[n]}\\
A_a&0&0&A^{[n]}.
\end{array}
\]
Set $T_{\mathrm{univ}}:=C_E^{\diff}+C_D^a-2A_a$.
By Lemma~\ref{lem:degree-one-hilbert-relation},
\[
(\pr_{\diff})_*T_{\mathrm{univ}}=E,\qquad
(\pr_b)_*T_{\mathrm{univ}}=C_D^{(n-1)},\qquad
(\pr_a)_*T_{\mathrm{univ}}=2F_{[n]}.
\]
Thus $T_{\mathrm{univ}}\in K_{\mathrm{univ}}$.

\begin{prop}
\label{prop:degree-one-universal-supporting-ray}
There exists a semiample $\QQ$-divisor $L_{\mathrm{univ}}$ on $\Xcal$ such that
\[
K_{\mathrm{univ}}\cap L_{\mathrm{univ}}^\perp
=
\RR_{\ge0}[T_{\mathrm{univ}}].
\]
\end{prop}

\begin{proof}
Choose $N_E,N_D\in\Nef(S)_\QQ$ such that
\[
\NEbar(S)\cap N_E^\perp=\RR_{\ge0}[E],
\qquad
\NEbar(S)\cap N_D^\perp=\RR_{\ge0}[D].
\]
Since $N_E-K_S=N_E+F$ and $N_D-K_S=N_D+F$ are ample, both divisors are semiample by the base-point-free theorem \cite[Theorem~3.3]{kollar1998birational}.
Put $Q_F:=nF^{[n]}-\frac12B^{[n]}$.
Then
\[
Q_F\cdot C_{E'}^{[n]}=1,
\qquad Q_F\cdot F_{[n]}=0,
\qquad Q_F\cdot A^{[n]}=1,
\]
for every $(-1)$-curve $E'$, and hence
$\NEbar(S^{[n]})\cap Q_F^\perp=\RR_{\ge0}[F_{[n]}]$.
Moreover,
\[
Q_F-K_{S^{[n]}}=(n+1)F^{[n]}-\frac12B^{[n]}
\]
has positive intersection with all the generators of $\NEbar(S^{[n]})$:
\[
(Q_F-K_{S^{[n]}})\cdot C_{E'}^{[n]}=2,
\quad
(Q_F-K_{S^{[n]}})\cdot F_{[n]}=1,
\quad
(Q_F-K_{S^{[n]}})\cdot A^{[n]}=1.
\]
Thus $Q_F$ is semiample, and likewise $N_D^{(n-1)}$ is semiample with
$\NEbar(S^{(n-1)})\cap(N_D^{(n-1)})^\perp
=\RR_{\ge0}[C_D^{(n-1)}]$.

Define
$L_{\mathrm{univ}}
:=\pr_{\diff}^*N_E+\pr_b^*N_D^{(n-1)}+\pr_a^*Q_F$.
Then $L_{\mathrm{univ}}\cdot T_{\mathrm{univ}}=0$.
For $\gamma\in K_{\mathrm{univ}}\cap L_{\mathrm{univ}}^\perp$,
\[
0
=
N_E\cdot(\pr_{\diff})_*\gamma
+
N_D^{(n-1)}\cdot(\pr_b)_*\gamma
+
Q_F\cdot(\pr_a)_*\gamma.
\]
Since each term is nonnegative, all three vanish.
Hence
\[
(\pr_{\diff})_*\gamma=aE,
\qquad
(\pr_b)_*\gamma=bC_D^{(n-1)},
\qquad
(\pr_a)_*\gamma=cF_{[n]}
\]
for some $a,b,c\ge0$.
Since $H^a=H^{\diff}+H^b$, we have
$cF=aE+bD$.
Substituting $D=2F-E$ gives $a=b$ and $c=2a$.
The three pushforwards of $\gamma$ agree with those of $aT_{\mathrm{univ}}$.
Hence their intersections with $H_i^{\diff}$, $H_i^b$, and $B^a$ agree.
For the last class,
$B^a\cdot\gamma=B^{[n]}\cdot(\pr_a)_*\gamma$.
Since these classes span $N^1(\Xcal)_\RR$, $\gamma=aT_{\mathrm{univ}}$.
\end{proof}

\begin{thm}
\label{thm:degree-one-universal-failure}
Let $S$ be a del Pezzo surface of degree $1$ and let $n\ge2$. 
Then
\[
\pr_{\diff}^*\Nef(S)
+\pr_b^*\Nef(S^{(n-1)})
+\pr_a^*\Nef(S^{[n]})
\subsetneq
\Nef(S^{[1,n]}).
\]
\end{thm}

\begin{proof}
Let $L_{\mathrm{univ}}$ be as in Proposition~\ref{prop:degree-one-universal-supporting-ray} and suppose that $T_{\mathrm{univ}}\in\NEbar(\Xcal)$.
Since $L_{\mathrm{univ}}$ is semiample and $L_{\mathrm{univ}}\cdot T_{\mathrm{univ}}=0$, choose an irreducible curve $\Gamma$ in a positive-dimensional fiber.
Then $[\Gamma]=aT_{\mathrm{univ}}$ for some $a>0$, and hence
\[
(\pr_{\diff})_*\Gamma=aE,
\qquad
(\pr_b)_*\Gamma=aC_D^{(n-1)},
\qquad
(\pr_a)_*\Gamma=2aF_{[n]}.
\]

For $n=2$, $(\pr_{\diff},\pr_b):S^{[1,2]}\to S\times S$ is the blow-up along the diagonal.
Let $\nu:C\to\Gamma$ be the normalization and set $f=\pr_{\diff}\circ\nu$, $g=\pr_b\circ\nu$.
Then $f_*[C]=aE$ and $g_*[C]=aD$, so $\Gamma$ is not contained in the exceptional divisor $\mathcal E$.
The argument of Lemma~\ref{lem:degree-one-collision-bound-hc} gives
\[
\mathcal E\cdot\Gamma\le E\cdot g_*[C]=3a.
\]
On the other hand, $B^a=2\mathcal E$, so
\[
\mathcal E\cdot\Gamma
=
\frac12B^a\cdot\Gamma
=
\frac a2B^{[2]}\cdot2F_{[2]}
=
4a,
\]
a contradiction.

For $n\ge3$, set $\Ycal:=S^{[n-1,n]}$ and consider
\[
\psi:\Ycal\longrightarrow\Xcal,
\qquad
(Z_{n-1}\subset Z_n)\longmapsto
(\res(Z_{n-1}\subset Z_n),Z_n).
\]
This morphism is projective and surjective: for $(p,Z_n)\in\Xcal$, the quotient
$\mathcal O_{Z_n,p}\twoheadrightarrow k(p)$
has kernel defining a length-$(n-1)$ subscheme $Z_{n-1}\subset Z_n$ with residual point $p$.
Choose an irreducible curve $\widetilde{\Gamma}\subset\Ycal$ dominating $\Gamma$ and let $r$ be its degree over $\Gamma$.
Then
\[
\res_*\widetilde{\Gamma}=arE,
\qquad
(\HC_{n-1}\circ p_b)_*\widetilde{\Gamma}=arC_D^{(n-1)},
\]
and Lemma~\ref{lem:degree-one-collision-bound-hc} gives
\[
\mathcal E_{n-1}\cdot\widetilde{\Gamma}\le(n+1)ar.
\]

Write
\[
(p_b)_*[\widetilde{\Gamma}]
=
\sum_{E'}x_{E'}C_{E'}^{[n-1]}
+yF_{[n-1]}
+cA^{[n-1]},
\qquad x_{E'},y,c\ge0.
\]
Since $(\HC_{n-1})_*C_{E'}^{[n-1]}=E'$,
$(\HC_{n-1})_*F_{[n-1]}=F$, and
$(\HC_{n-1})_*A^{[n-1]}=0$, applying $\HC_{n-1}$ gives
\[
arD
=
\sum_{E'}x_{E'}E'+yF.
\]
Since $\RR_{\ge0}[D]$ is an extremal ray of $\NEbar(S)$, the decomposition on the right can contain only the class $D$.
Hence $y=0$, $x_{E'}=0$ for $E'\ne D$, and $x_D=ar$.
Therefore
\[
(p_b)_*[\widetilde{\Gamma}]
=
arC_D^{[n-1]}+cA^{[n-1]}.
\]
Also $(p_a)_*[\widetilde{\Gamma}]=2arF_{[n]}$.
Using $2\mathcal E_{n-1}=B^a-B^b$,
\[
\mathcal E_{n-1}\cdot\widetilde{\Gamma}
=
\frac12\left(
B^{[n]}\cdot2arF_{[n]}
-
B^{[n-1]}\cdot
(arC_D^{[n-1]}+cA^{[n-1]})
\right)
=
(n+2)ar+c
\ge
(n+2)ar,
\]
contradicting the upper bound above.
Hence $T_{\mathrm{univ}}\notin\NEbar(\Xcal)$.
\end{proof}

\bibliographystyle{plain}
\bibliography{bib}

\end{document}